\documentclass[11pt,twoside]{amsart}
\usepackage[russian,english]{babel}
\Russian

\usepackage{amssymb}
\usepackage{latexsym}

\usepackage{color}

\catcode`\@=11
\def\omathop#1#2#3{\let\temp=#1\def\letter{#2}
  \ifcat#3_ \let\next\@@olim\else\let\next\@olim\fi\next#3}
\def\@olim{\letter\text{-}\!\temp}
\def\@@olim_#1{\mathchoice{
   \setbox0=\hbox{$\displaystyle\letter\text{-}\!\temp\!\text{-}\letter$}
   \setbox2=\hbox{$\displaystyle\temp$}
   \setbox4=\hbox{$\scriptstyle#1$}
   \dimen@=\wd4 \advance\dimen@ by -\wd2 \divide\dimen@ by2
   \def\next{\letter\text{-}\!\temp_{\hbox to 0pt{\hss$\scriptstyle#1$\hss}}
     \hskip\dimen@}
   \ifdim\wd2>\wd4 \def\next{\@olim_{#1}}\fi
   \ifdim\wd4>\wd0 \def\next{\mathop{\llap{$\letter$-}\!\temp}\limits_{#1}}\fi
   \next}
   {\@olim_{#1}}{\@olim_{#1}}{\@olim_{#1}}}

\def\olim{\omathop{\lim}{o}}

\catcode`\@=13

\theoremstyle{plain}
\newtheorem{thm}{Theorem}[section]

\newtheorem{cor}[thm]{Corollary}
\newtheorem{lemma}[thm]{Lemma}
\newtheorem{rem}[thm]{Remark}

\newtheorem{example}[thm]{Example}

\theoremstyle{definition}
\newtheorem{definition}[thm]{Definition}

\numberwithin{equation}{section}

\begin{document}

\title{Decomposition of an abstract    Uryson operator}

\author{M.~A.~Ben Amor }

\address{Research Laboratory of Algebra, Topology, Arithmetic, and Order\\ Department of Mathematics \\ Faculty of Mathematical, Physical and Natural Sciences of Tunis \\ Tunis-El Manar University, 2092-El Manar, Tunisia}

%\email{mohamedamine.benamor@ipest.rnu.tn}

\title{Decomposition of an abstract    Uryson operator}

\author{M.~Pliev }

\address{South Mathematical Institute of the Russian Academy of Sciences\\
str. Markusa 22,
Vladikavkaz, 362027 Russia}

%\email{maratpliev@gmail.com}

%\centerline{\today}



\begin{abstract}
We consider  the space of abstract  Uryson operators firstly introduced in \cite{Maz-1}. We obtain the formulas for   band projections on the band generated by increasing set of a positive Uryson operators and  on the band generated one-dimensional  abstract Uryson operators.  We also calculate the laterally continuous part of a abstract  Uryson operator.
\end{abstract}

\maketitle

%\tableofcontents

\section{Introduction}

Nowadays the theory of regular operators in vector lattices is a great area of Functional Analysis {\cite{Ab,Al,Ku,Za}}. Nonlinear maps between vector lattices constitue  a more delicate  subject. The interesting class of nonlinear maps which called abstract Uryson  operators   was introduced and studied in 1990 by Maz\'{o}n and Segura de Le\'{o}n \cite{Maz-1,Maz-2}, and then considered to be defined on lattice-normed spaces by Kusraev and second named author \cite{Ku-1,Ku-2,Pl-3}. The space of all abstract Uryson operators has a nice order properties, but the structure of this space is still less known.
In this  notes we investigate  some bands in the space of all abstract Uryson operators and find formulas for band projections on this bands.

\section{Preliminary information}

The  goal of this section is to introduce some basic definitions and facts. General information on vector lattices  the reader can find in the books \cite{Al,Ku,Za}.

\begin{definition} \label{def:ddmjf0}
Let $E$ be a vector lattice, and let $F$ be a real linear space. An operator $T:E\rightarrow F$ is called \textit{orthogonally additive} if $T(x+y)=T(x)+T(y)$ whenever $x,y\in E$ are disjoint.
\end{definition}

It follows from the definition that $T(0)=0$. It is immediate that the set of all orthogonally additive operators is a real vector space with respect to the natural linear operations.

\begin{definition}
Let $E$ and $F$ be vector lattices. An orthogonally additive operator $T:E\rightarrow F$ is called:
\begin{itemize}
  \item \textit{positive} if $Tx \geq 0$ holds in $F$ for all $x \in E$;
  \item \textit{order bounded} it $T$ maps order bounded sets in $E$ to order bounded sets in $F$.
\end{itemize}
An orthogonally additive order bounded operator $T:E\rightarrow F$ is called an \textit{abstract Uryson} operator.
\end{definition}

The set of all abstract Uryson operators from $E$ to $F$ we denote by $\mathcal{U}(E,F)$. We will consider some examples.
The most famous one is the nonlinear integral Uryson operator.

\begin{example}\label{Ex-0}
Let $(A,\Sigma,\mu)$ and $(B,\Xi,\nu)$ be $\sigma$-finite complete measure spaces, and let $(A\times B,\mu\times\nu)$ denote the completion of their product measure space. Let $K:A\times B\times\Bbb{R}\rightarrow\Bbb{R}$ be a function satisfying the following conditions\footnote{$(C_{1})$ and $(C_{2})$ are called the Carath\'{e}odory conditions}:
\begin{enumerate}
  \item[$(C_{0})$] $K(s,t,0)=0$ for $\mu\times\nu$-almost all $(s,t)\in A\times B$;
  \item[$(C_{1})$] $K(\cdot,\cdot,r)$ is $\mu\times\nu$-measurable for all $r\in\Bbb{R}$;
  \item[$(C_{2})$] $K(s,t,\cdot)$ is continuous on $\Bbb{R}$ for $\mu\times\nu$-almost all $(s,t)\in A\times B$.
\end{enumerate}
Given $f\in L_{0}(A,\Sigma,\mu)$, the function $|K(s,\cdot,f(\cdot))|$ is $\mu$-measurable  for $\nu$-almost all $s\in B$ and $h_{f}(s):=\int_{A}|K(s,t,f(t))|\,d\mu(t)$ is a well defined and $\nu$-measurable function. Since the function $h_{f}$ can be infinite on a set of positive measure, we define
$$
\text{Dom}_{A}(K):=\{f\in L_{0}(\mu):\,h_{f}\in L_{0}(\nu)\}.
$$
Then we define an operator $T:\text{Dom}_{A}(K)\rightarrow L_{0}(\nu)$ by setting
$$
(Tf)(s):=\int_{A}K(s,t,f(t))\,d\mu(t)\,\,\,\,\nu-\text{a.e.}\,\,\,\,(\star)
$$
Let $E$ and $F$ be order ideals in $L_{0}(\mu)$ and $L_{0}(\nu)$ respectively, $K$ a function satisfying $(C_{0})$-$(C_{2})$. Then $(\star)$ defines an \textit{orthogonally additive order bounded integral operator} acting from $E$ to $F$ if $E\subseteq \text{Dom}_{A}(K)$ and $T(E)\subseteq F$.
\end{example}

\begin{example} \label{Ex-1}
We consider the vector space $\mathbb R^m$, $m \in \mathbb N$ as a vector lattice with the coordinate-wise order: for any $x,y \in \mathbb R^m$ we set $x \leq y$ provided $e_i^*(x) \leq e_i^*(y)$ for all $i = 1, \ldots, m$, where $(e_i^*)_{i=1}^m$ is the coordinate functionals on $\mathbb R^m$. Let $T:\Bbb{R}^{n}\rightarrow\Bbb{R}^{m}$. Then $T\in\mathcal{U}(\Bbb{R}^{n},\Bbb{R}^{m})$ if and only if there are real functions $T_{i,j}:\Bbb{R}\rightarrow\Bbb{R}$,
$1\leq i\leq m$, $1\leq j\leq n$ satisfying $T_{i,j}(0)=0$ such that
$$
e_i^*\bigl(T(x_{1},\dots,x_{n})\bigr) = \sum_{j=1}^{n}T_{i,j}(x_{j}),
$$
In this case we write $T=(T_{i,j})$.
\end{example}

Let $E$ be a vector lattice and $x\in E$. Recall that an element $z\in E$ is called a {\it component} or a \textit{fragment} of $x$ if  $z\bot(x-z)$. The set of all  fragments of an element $x$ is denoted by $\mathcal{F}_{x}$. The notations $z\sqsubseteq x$ means that $z$ is a fragment of $x$.
Consider the following order in $\mathcal{U}(E,F):S\leq T$ whenever $T-S$ is a positive operator. Then $\mathcal{U}(E,F)$
becomes an ordered vector space.  If vector lattice $F$ is Dedekind complete we have the following theorem.
\begin{thm}(\cite{Maz-1},Theorem~3.2)\label{th-1}.
Let $E$ and $F$ be a vector lattices, $F$ Dedekind complete. Then $\mathcal{U}(E,F)$ is a Dedekind complete vector lattice. Moreover for $S,T\in \mathcal{U}(E,F)$ and for $f\in E$ following hold
\begin{enumerate}
\item~$(T\vee S)(f):=\sup\{Tg+Sh:\,f=g+h;\,g\bot h\}$.
\item~$(T\wedge S)(f):=\inf\{Tg+Sh:\,f=g+h;\,g\bot h\}.$
\item~$(T)^{+}(f):=\sup\{Tg:\,g\sqsubseteq f\}$.
\item~$(T)^{-}(f):=-\inf\{Tg:\,g;\,\,g\sqsubseteq f\}$.
\item~$|Tf|\leq|T|(f)$.
\end{enumerate}
\end{thm}

\section{Bands, generated by a subsets of an operators}

Order projections are an important tool in the vector lattice theory. There are a some interesting results concerning order projections in the spaces of linear, bilinear  and orthogonally additive operators in vector lattices \cite{Al-1,Ko,Ko-1,Pag,Pl-3, Pl-4,Pl-5}. In this section we find projection formulas for order projections  on a band, generated by increasing set of positive abstract Uryson operators and band, generated by operators of the finite rank.

\begin{lemma}\label{disjoint}
Let $E,F$ be vector lattices, with $F$ Dedekind complete. Two operators $T,S\in\mathcal{U}_{+}(E,F)$  are disjoint if and only if for arbitrary $e\in E$ and $\varepsilon>0$
there exist a partition of unity $(\rho_{\alpha})$  and a family $(e_{\alpha})\subset\mathcal{F}_{e}$, such that
$$
\rho_{\alpha}Te_{\alpha}\leq\varepsilon Te;\,
\rho_{\alpha}S(e-e_{\alpha})\leq\varepsilon Se;\,
$$
\end{lemma}

\begin{proof}
Let $S\wedge T=0$. Fix arbitrary $e\in E$ and consider the element
$f=Se\wedge Te+\rho_{Se\wedge Te}^{\bot}(Se+Te)$, where $\rho_{g}$ is a
projection on the band $\{g\}^{\bot\bot}$.  Then we have $Se+Te\in\{f\}^{\bot\bot}$ and
$\rho_{Se}f\leq Se$ and  $\rho_{Te}f\leq Te$. Then, using the disjointness of $T$ and $S$, we may write, $\rho_{\alpha}(Te_{\alpha}+S(e-e_{\alpha}))\leq\varepsilon f$  for a certain partition of unity $(\rho_{\alpha})$ and a certain set of a fragments of $e$.  Consequently,
$$
\rho_{\alpha}Te_{\alpha}\leq\rho_{Te}\varepsilon f\leq\varepsilon Te;\,
\rho_{\alpha}S(e-e_{\alpha})\leq\rho_{Se}\varepsilon f\leq\varepsilon Se.
$$
The converse statement is obvious.
\end{proof}

For each set $A\subset\mathcal{U}(E,F)$ we denote by $\pi_{A}$ the projection in the $\mathcal{U}(E,F)$ on the band
$\{A\}^{\bot\bot}$ and set $\sigma_{A}=(\pi_{A})^{\bot}$.  A the set $A\subset\mathcal{U}_{+}(E,F)$  is called { \it increasing} if for arbitrary
$S,T\in A$  there exists a $Q\in A$ such that $S,T\leq Q$.

\begin{thm}\label{proj-1}
Let $E,F$ be a vector lattices, $F$ Dedekind complete, $A\subset\mathcal{U}_{+}(E,F)$ be a increasing set.
Then the following equation holds for
an arbitraries $T\in\mathcal{U}_{+}(E,F)$ and  $e\in E$
$$
\sigma_{A}Te=\inf\limits_{\varepsilon>0,S\in A}\sup
\{\rho Tf:\,\rho Sf\leq\varepsilon Se;\,f\in\mathcal{F}_{e};\,\rho\in\mathfrak{B}(F)\}.
$$
\end{thm}

\begin{proof}
Denote the right-hand side of the above equation  by $\vartheta(T)(e)$ and let $\kappa(T)=T-\vartheta(T)$. It is clear that map $\vartheta(T):E\to F$ is orthogonally additive, order bounded  operator and $0\leq\vartheta(T)\leq T$. It is enough to prove that $\kappa(T)=\pi_{A}T$. Now we may write
$$
\kappa(T)(e)=\sup\limits_{\varepsilon>0,S\in A}\inf\{\rho Tf+\rho^{\bot}Te:\,\rho Sf\leq\varepsilon Se;\,f\in\mathcal{F}_{e};\,\rho\in\mathfrak{B}(F)\}.
$$
By the condition, $\pi_{A}T\in\{A\}^{\bot\bot}$, there  exists  a net
of  operators $(T_{\gamma})_{\gamma\in\Gamma}\subset\mathcal{U}_{+}(E,F)$, so that  $T_{\gamma}\uparrow\pi_{A}T$ and
$$
T_{\gamma}\leq\sum\limits_{i=1}^{n(\gamma)}\lambda_{i}|S_{i}|:\, S_{i}\in A,\,n(\gamma)\in\Bbb{N};\,\gamma\in\Gamma.
$$
Using the fact, that $A$ is increasing set we have
$$
(T_{\gamma})\subset\bigcup\{[0,nS]:\, S\in A,\,n\in\Bbb{N}\}.
$$
Fix $\gamma_{0}\in\Gamma$, $T_{\gamma_{0}}\in(T_{\gamma})_{\gamma\in\Gamma}$. Then $T_{\gamma_{0}}\leq nS$ for certain $S\in A$ and $n\in\Bbb{N}$.
For an arbitrary  $\varepsilon>0$ we choose elements  $\rho\in\mathfrak{B}(F)$  and  $f\in\mathcal{F}_{e}$
such that $\rho S(e-f)\leq\varepsilon Se$. Hence
$$
T_{\gamma_{0}}e\leq\rho T_{\gamma_{0}}(e-f)+\rho Tf+\rho^{\bot}Te\leq
$$
$$
\leq\rho nS(e-f)+\rho T(e-f)+\rho Tf\leq
$$
$$
\leq\varepsilon nSe+\rho T(e-f)+\rho Tf.
$$
Then we have
$$
T_{\gamma_{0}}e\leq\varepsilon nSe+
\inf\{\rho Tf+\rho^{\bot}Te;\,f\in\mathcal{F}_{e};\,\rho Sf\leq\varepsilon Se;\,\rho\in\mathfrak{B}(F)\}\leq
$$
$$
\leq\varepsilon nSe+\kappa(T)(e).
$$
Since  the  $\varepsilon>0$ is  arbitrary,  it  follows  that
$T_{\gamma_{0}}e\leq\kappa(T)(e)$. Therefore $\sup_{\gamma\in\Gamma}T_{\gamma}e\leq\kappa(T)(e)$  and
$$
\pi_{A}Te=\sup\limits_{\gamma}T_{\gamma}e\leq\kappa(T)(e)\Rightarrow\vartheta(T)(e)\leq\sigma_{A}Te.
$$
Since $e\in E$ is arbitrary, it follows that $\vartheta(T)\leq\sigma_{A}T$.
Let us   prove  the  reverse  assertion.  It  is  not difficult to check   that
$$
\vartheta(\sigma_{A}T)\leq\vartheta(T)\leq\vartheta(\sigma_{A}T)+\vartheta(\pi_{A}T)
$$
for each $T\in\mathcal{U}_{+}(E,F)$. On  the  other  hand,  we have that
$$
\vartheta(\pi_{A}T)\leq\sigma_{A}\pi_{A}T=0.
$$
Therefore $\vartheta(T)=\vartheta(\sigma_{A}T)$. Now,  it  remains  to  prove  that  $\vartheta(\sigma_{A}T)=\sigma_{A}Tx$.
Let $C=\sigma_{A}T$ and $S\in A$. Then $C\bot S$ and by \ref{disjoint}, we can choose a full family  $(\rho_{\xi})_{\xi\in\Xi}$ of mutually disjoint order projections on $F$ and  family  $(e_{\xi})_{\xi\in\Xi}$ of a fragments of $e$ so that
$$
\varepsilon Ce\leq\rho_{\xi}Ce_{\xi}=\rho_{\xi}(\rho_{\xi}Ce_{\xi}+\rho_{\xi}^{\bot}Ce)\geq
$$
$$
\geq\rho_{\xi}\inf\{\rho Cf+\rho^{\bot}Ce:\,\rho\in\mathfrak{B}(F),\,f\in\mathcal{F}_{e},\,\rho Sf
\leq\varepsilon Se\}.
$$
Hence $\varepsilon Ce\geq\sup_{\xi}\rho_{\xi}\kappa(C)$ for an arbitrary $\varepsilon>0$.
It means that $Ce-\vartheta(C)=\kappa(C)=0$.
\end{proof}

\begin{cor}\label{cor:2}
For  arbitrary  $S,T\in\mathcal{U}_{+}(E,F)$ and $e\in E$
$$
\pi_{S}Tx=\sup\limits_{\varepsilon>0}\inf\{\rho Tf+\rho^{\bot}Te:\,\rho S(e-f)\leq\varepsilon Se\};
$$
$$
\sigma_{S}Te=\inf\limits_{\varepsilon>0}\sup\{\rho Tf:\,\rho Sf\leq\varepsilon Se\}.
$$
where $\rho\in\mathfrak{B}(F)$ and $f\in\mathcal{F}_{e}$.
\end{cor}

Observe that $\rho_{Se}^{\bot}(\sigma_{S}Te)=\rho_{Se}^{\bot}Te$  and $\rho_{Se}(\pi_{S}Te)=\pi_{S}Te$,
where  $\rho_{Se}$  is the projection  on the band  $\{Se\}^{\bot\bot}$.  In particular,
$$
\sigma_{S}Te=\rho_{Se}Te+\inf\limits_{\varepsilon>0}\sup\{\rho Tf:\,\rho Sf\leq\varepsilon Se;\,
\rho\in[0,\rho_{Se}];\,f\in\mathcal{F}_{e}\}.
$$
Let $E,F$ be a vector lattices and $\varphi\in\mathcal{U}(E,\Bbb{R})$, $u\in F$. One-dimensional abstract Uryson operator $S:E\rightarrow F$, $Se=u\varphi(e)$ is denoted by $\varphi\otimes u$.

\begin{lemma}\label{lem:6}
Let $E,F$ be the same as in \ref{proj-1} and let $S=\varphi\otimes u$ be one-dimensional positive abstract Uryson operator, where $\varphi\in\mathcal{U}_{+}(E,\Bbb{R})$, $u\in F_{+}$.
Then the following formula valid
\begin{align*}
\pi_{\varphi\otimes u}Te=\sup\limits_{\varepsilon>0}\inf\limits_{f}\{\rho_{u} Te:\,\rho S(e-f)\leq\varepsilon \varphi(e),\,f\in\mathcal{F}_{e}\}.
\end{align*}
\end{lemma}

\begin{proof}
Let $\rho_{u}$ order projection on the band $\{u\}^{\bot\bot}$ and $\rho\in[0,\rho_{u}]$. Then by conditions
$\rho Sf\leq\varepsilon Se$ we have $\varphi(f)\rho u\leq\varepsilon\varphi(e)u$. This inequality is equivalent the $\varphi(f)\leq\varepsilon\varphi(e)$. We observe that for $e\in E$ so that $\varphi(e)\neq 0$ we have $\rho_{Se}=\rho_{u}$ and conditions $\rho Se\leq\varepsilon Se$ do not depends of $\rho$. Thus  the  supremum   is attained  for $\rho=\rho_{u}$. Finally we have
$$
\sigma_{S}Te=\rho_{u}^{\bot}Te+
\inf_{\varepsilon>0}\sup\{\rho_{u}Tf:\,\varphi(f)\leq\varepsilon\varphi(e);\,f\in\mathcal{F}_{e}\}.
$$
\end{proof}

\begin{cor}\label{cor:4}
Let $H$ be a band generated one-dimensional  abstract Uryson operators.
Then for an arbitrary $T\in\mathcal{U}_{+}(E,F)$ and all $e\in E$ the following formula valid
\begin{align*}
\pi_{H}Te=\sup\limits_{\varepsilon>0,\varphi\in\mathcal{U}_{+}(E,\Bbb{R})}\inf\limits_{f}\{\rho_{u} Tx:\,\rho S(e-f)\leq\varepsilon \varphi(e),\,f\in\mathcal{F}_{e}\}.
\end{align*}
\end{cor}

\section{Laterally continuous part }

In this section we find formulas for calculating the laterally continuous parts  of an abstract Uryson operators.
Recall that a net $(x_\alpha)$ in a vector lattice $E$ \textit{laterally converges} to $x \in E$ if $x_\alpha \sqsubseteq x_\beta \sqsubseteq x$ for all indices $\alpha < \beta$ and $x_\alpha \overset{\rm o}\longrightarrow x$. In this case we write $x_\alpha \overset{\rm lat}\longrightarrow x$. For positive elements $x_\alpha, x$ the condition $x_\alpha \overset{\rm lat}\longrightarrow x$ means that $x_\alpha \sqsubseteq x$ and $x_\alpha \uparrow x$.

\begin{definition}
Let $E,F$ be vector lattices. An orthogonally additive operator $T:E\to X$ is called {\it laterally continuous}
if $T$  sends laterally convergent nets in $E$ to order convergent nets in $F$.
\end{definition}

The vector space of all laterally continuous abstract Uryson operators from $E$ to $F$ we denote by $\mathcal{U}_{n}(E,F)$. If $E,F$ are vector lattices with $F$ Dedekind complete the $\mathcal{U}_{n}(E,F)$
is a band in $\mathcal{U}(E,F)$ (\cite{Maz-1}, Proposition~3.8). Therefore every abstract Uryson operator $T$ has a unique decomposition $T=T_{n}+T_{s}$, where $T_{n}\in\mathcal{U}_{n}(E,F)$ and $T_{s}\in\mathcal{U}_{n}(E,F)^{\bot}$.
Operator $T_{n}$ is called a {\it  laterally continuous part} of $T$. The order projection on the band $\mathcal{U}_{n}(E,F)$ in the space $\mathcal{U}(E,F)$ we denote by $\pi_{n}$.

\begin{definition}\label{def:adm}
Let $D$ be a subset of a vector lattice $E$. The subset $D$ is called \it{admissible} if the following conditions hold
\begin{enumerate}
\item~if $x\in D$, then $y\in D$ for every $y\in\mathcal{F}_{x}$;
\item~if $x,y\in D$, $x\bot y$ then $x+y\in D$.
\end{enumerate}
\end{definition}

\begin{example}\label{adm-1}
Let $E$ be a vector lattice. Every order ideal in $E$ is an admissible set.
\end{example}

\begin{example}\label{adm-1}
Let $E$ be a vector lattice and $e\in E$. Then $\mathcal{F}_{e}$  is an admissible set.
\end{example}

\begin{example}\label{adm-1}
Let $E,F$ be a vector lattices and $T\in\mathcal{U}_{+}(E,F)$. Then $\mathcal{N}_{T}:=\{e\in E:\,T(e)=0\}$  is an admissible set.
\end{example}

Let $E,F$ be  vector lattices, with $F$ Dedekind complete, $T\in\mathcal{U}_{+}(E,F)$ and $D\subset E$ is an admissible set. Then we may
define a map  $\pi^{D}T:E\rightarrow F_{+}$ by formula
$$
\pi^{D}T(x)=\sup\{Ty:\,y\sqsubseteq x,\,y\in D\},
$$
for every $x\in E$.

\begin{lemma}\label{le:01}
Let $E,F$ be vector lattices  with $F$ Dedekind complete,  $T\in\mathcal{U}_{+}(E,F)$ and $D$ is an admissible set.
Then $\pi^{D}T$ is a positive abstract Uryson operator and $\pi^{D}T\in\mathcal{F}_{T}$.
\end{lemma}

\begin{proof}
Let  show that $\pi^{D}(T)\in\mathcal{U}_{+}(E,F)$.
Fix $x,y\in E$, $x\bot y$. If $z\sqsubseteq x+y$ and $z\in D$ by Riesz decomposition property,
there exist $z_{1},z_{2}$ such that $z_{1}+z_{2}=z$, $z_{1}\bot z_{2}$ and $z_{1},z_{2}\in D$. Then we have
$$
Tz=T(z_{1}+z_{2})=T(z_{1})+T(z_{2}),
$$
and therefore $\pi^{D}(T)(x+y)\leq\pi^{D}(T)(x)+\pi^{D}(T)(y)$. Let us to prove a reverse inequality.
If $z_{1}\sqsubseteq x$, $z_{1}\in D$ and  $z_{2}\sqsubseteq y$, $z_{1}\in D$ we have $z_{1}+z_{2}\in D$. Hence
$$
T(z_{1})+T(z_{2})=T(z_{1}+z_{2})\leq\pi_{D}(T)(x+y).
$$
Take a supremum in the left hand we may write
$$
\pi^{D}(T)(x)+\pi^{D}(T)(y)\leq\pi^{D}(T)(x+y).
$$
Finally we have
$$
\pi^{D}(T)(x)+\pi^{D}(T)(y)=\pi^{D}(T)(x+y).
$$
Now fix an order projection $\rho$  on the
Dedekind complete vector lattice space $F$.
Then the operator $\Pi:\mathcal{U}(E,F)\rightarrow\mathcal{U}(E,F)$ define by $\Pi(T)=\rho\pi^{D}T$
satisfies $0\leq\Pi(T)\leq T$ for each $T\in\mathcal{U}_{+}(E,F)$ and $\Pi^{2}=\Pi$. Thus, by (\cite{Al},Theorem~$1.44$)
the operator $\Pi$ is an order projection on $\mathcal{U}(E,F)$. Consequently,
if $T$ is a positive abstract Uryson operator, then for  each admissible set $D\subset E$
the operator $\pi^{D}T$ is a fragment  of $T$.
\end{proof}

If $D=\mathcal{F}_{x}$ we shall denote
operator $\pi^{D}T$ by $\pi^{x}T$.

\begin{definition}
An admissible set $D$  is
called    {\it laterally   dense } if for every $e\in E$  there exists a  laterally convergence net  $(e_{\alpha})_{\alpha\in\Lambda}\subset D$  such  that
$e_\alpha \overset{\rm lat}\longrightarrow e$. The set of all admissible  subsets (admissible  laterally dense subsets)  of $E$ we denote by $AD(E)$ ($ADD(E)$).
\end{definition}

Now we are ready to formulate the main result of this section.

\begin{thm}\label{projection}
Let $E,F$ be vector lattices  with $F$ Dedekind complete.
Then laterally continuous part of an abstract Uryson operator $T:E\to F$ may be calculated by formula $T_{n}=\pi_{n}T$ where $\pi_{n}=\inf\{\pi^{D}:\,D\in ADD(E)\}$.
\end{thm}

First we need  some auxiliary notions.

\begin{definition}
An abstract Uryson operator $T:E\to F$  is called {\it singular}    if it  vanishes on  some  laterally   dense  admissible set.  The  set of   all  singular   operators  is denoted  by $\mathcal{U}_{s}(E,F)$.
\end{definition}

Observe that every order dense ideal of the $E$ is a laterally dense admissible set.

\begin{thm}\label{singular}
Let $E,F$ be the same as in Theorem~\ref{projection}.
Then $\mathcal{U}_{s}(E,F)^{\bot}=\mathcal{U}_{n}(E,F)$.
\end{thm}

\begin{proof}
Let $T$ be a laterally continuous operator and $T\notin\mathcal{U}_{s}(E,F)^{\bot}$. Then there exists $S\in\mathcal{U}_{s}(E,F)^{\bot}$ such that $R:=T\wedge S>0$. Thus  $R$ vanishes on some laterally dense admissible set and since $0\leq R\leq T$, we have $R$ is a laterally continuous abstract Uryson operator. Therefore $R=0$ and $T\in\mathcal{U}_{s}(E,F)^{\bot}$. Let us prove  the reverse inclusion. Let $T\in\mathcal{U}_{s}(E,F)^{\bot}$, $T\geq 0$ and assume that $T$ is not laterally continuous. Then there exits  an element $e\in E$ and laterally convergence $(e_{\alpha})_{\alpha\in\Lambda}\subset E$, $e=\olim_{\alpha}e_{\alpha}$, such that $\olim_{\alpha}Te_{\alpha}=\sup_{\alpha}Te_{\alpha}<Te$. Denote $f_{\alpha}=e-e_{\alpha}$. Observe that
$$
e=f_{\alpha}+e_{\alpha};\,f_{\alpha}\bot e_{\alpha};\,
Te=T(f_{\alpha}+e_{\alpha})=Tf_{\alpha}+Te_{\alpha};
$$
$$
\alpha\in\Lambda;\,
T(e-e_{\alpha})=Tf_{\alpha}=Te-Te_{\alpha}.
$$
Consider the net of the fragments $S_{\alpha}$ of the operator $T$ defined by formula
$$
S_{\alpha}:=\pi^{f_{\alpha}}T;\,\alpha\in\Lambda
$$
and take the operator $S=\inf_{\alpha}S_{\alpha}$. The operator $S:E\to F$ is nonzero  positive laterally continuous abstract Uryson operator. Indeed
$$
Se=\olim_{\alpha}S_{\alpha}e=\olim_{\alpha}Tf_{\alpha}=
\olim_{\alpha}T(e-e_{\alpha})=Te-\olim_{\alpha}Te_{\alpha}>0.
$$
Now we may prove that   $S$ is a singular operator. Remark that $Sg=0$ for every $g\in E$, $g\bot f_{\alpha_{0}}$ for
some $\alpha_{0}\in\Lambda$. Denote the set $\{g\in E:\,Sg=0\}$ by $\widetilde{E}$. It is clear that $\widetilde{E}$ is a order ideal. Let us prove that $\widetilde{E}$ is order dense. Indeed, assume that there is $f\in E_{+}$, so that
$$
0<f\wedge|f_{\alpha}|,\,\text{for every}\,\,\alpha\in\Lambda.
$$
But $\olim_{\alpha}f_{\alpha}=0$, it is a contradiction.
\end{proof}

A family $\mathfrak{A}$ of an admissible subsets of $E$ is called {\it upward saturated} if for every $A\in\mathfrak{A}$, $B\in AD(E)$, $B\supset A$ we have $B\in\mathfrak{A}$. By $\mathcal{U}_{\mathfrak{A}}(E,F)$
denote the set $\{T\in\mathcal{U}(E,F):\,\mathcal{N}_{T}\in\mathfrak{A}\}$.

\begin{lemma}\label{satur}
Let $E,F$ be the same as in Theorem~\ref{projection} and  $\mathfrak{A}$ be a family of an upward saturated admissible subsets of $E$. Then
$\pi^{\mathfrak{A}}:=\inf\{\pi^{D}:\,D\in\mathfrak{A}\}$ is a band projection to $\mathcal{U}_{\mathfrak{A}}(E,F)^{\bot}$.
\end{lemma}

\begin{proof}
It is clear that $\pi^{\mathfrak{A}}$  is some band projection in $\mathcal{U}(E,F)$.  Denote by $\rho$ the projection to $\mathcal{U}_{\mathfrak{A}}(E,F)^{\bot}$.
If $T\in\mathcal{U}_{\mathfrak{A}}(E,F)$
then  $\pi^{\mathfrak{A}}T\leq\pi^{\mathcal{N}_{T}}T=0$. Hence $\pi^{\mathfrak{A}}\leq\rho$.
Conversely, suppose  that  $T\in\mathcal{U}_{\mathfrak{A}}(E,F)^{\bot}$.  For
every admissible set $D\in\mathfrak{A}$ the operator  $T-\pi^{D}T$ vanishes
on  $D$.  Consequently,
$T-\pi^{D}T\in\mathcal{U}_{\mathfrak{A}}(E,F)\cap\mathcal{U}_{\mathfrak{A}}(E,F)^{\bot}=0$.
Thus,  $T=\pi^{D}T$ and using the fact that $D$ is arbitrary, we have $T=\pi^{\mathfrak{A}}T$ and  therefore $\pi^{\mathfrak{A}}\geq\rho$.
\end{proof}

\begin{cor}[Theorem~\ref{projection}]
Let $E,F$ be a same as in Theorem~\ref{projection} . Then the projection
$$
\pi_{n}=\inf\{\pi^{D}:\,D\in ADD(E)\}
$$
is  the  respective  band  projections  in  the  space $\mathcal{U}(E,F)$  to
$\mathcal{U}_{s}(E,F)^{\bot}$.
\end{cor}

\begin{cor}
Let $E,F$ be a same as in Theorem~\ref{projection} . Then the following formulas valid for every $T\in\mathcal{U}_{+}(E,F)$:
$$
\pi_{n}T(e)=\inf\{\sup_{\alpha}Te_{\alpha}:\,e_\alpha \overset{\rm lat}\longrightarrow e\}, \,(e\in E).
$$

\begin{proof}
Put
$R(T,e)=\inf\{\sup_{\alpha}Te_{\alpha}:\,e_\alpha \overset{\rm lat}\longrightarrow e\}$. Given a net $(e_{\alpha})$,
$e_\alpha \overset{\rm lat}\longrightarrow e$, consider the family $ADD(e_{\alpha})$ admissible laterally dense subsets of $E$, such that $(e_{\alpha})\subset D,\,D\in ADD(e_{\alpha})$. Then we have $\sup_{\alpha}Te_{\alpha}\leq\pi^{D}Te$ for every $D\in ADD(e_{\alpha})$. Consequently,
$$
R(T,e)\leq\inf\{\pi^{D}Te:\,D\in\bigcup\limits_{e_\alpha \overset{\rm lat}\longrightarrow e} ADD(e_{\alpha})\}.
$$
On  the other  hand,  every $D$  laterally  dense admissible subset of $E$   includes  some
net $(e_{\alpha})$, $e_\alpha \overset{\rm lat}\longrightarrow e$ i.e.,  $D\in ADD(e_{\alpha})$.  Thefore
$R(T,e)\leq\pi_{n}Te$.   Hence, $0\leq R(T-\pi_{n}T,e)\leq\pi_{n}(T-\pi_{n}T)e=0$.  Moreover,  $R(\pi_{n}T,e)=\pi_{n}Te$, because $\pi_{n}Te=\sup_{\alpha}\pi_{n}Te_{\alpha}$ for
every net $(e_{\alpha})$, $e_\alpha \overset{\rm lat}\longrightarrow e$.  Since the map $R(T,e)$ is  additive
in the variable $T$,  we obtain
$$
R(T,e)=R(\pi_{n}T,e)+R(T-\pi_{n}T,e)=\pi_{n}Te
$$
\end{proof}
\end{cor}

\begin{rem}
The assertions \ref{projection} and \ref{singular} were proven in \cite{Pl-4} in assumptions that vector lattice $E$ has the principal projection property.
\end{rem}


\begin{thebibliography}{8}

\bibitem{Ab} \textsc{Abramovich\,Y.~A., Aliprantis\,C.~D.} \emph{ An Invitation to
Operator Theory.}---AMS, 2002.


\bibitem{Al} \textsc{C.~D.~Aliprantis}, \textsc{O.~Burkinshaw}, \emph{Positive Operators}, Springer, Dordrecht. (2006).

\bibitem{Al-1}\textsc{C.~D.~Aliprantis}, \textsc{O.~Burkinshaw}, \emph{The components of a positive operator}, Math. Z., 184(2) (1983), pp.~245-257.

\bibitem{Ko}\textsc{E.~V.~Kolesnikov},  \emph{Decomposition  of a positive operator}, Siberian Math. J., 30(5) (1989), pp.~77-79.

\bibitem{Ko-1}\textsc{E.~V.~Kolesnikov},  \emph{Several order projections generated by ideals of a vector lattice}, Siberian Math. J., 36(6) (1995), pp.~1342-1349.


\bibitem{Ku} \textsc{A.~G.~Kusraev}, \emph{Dominated Operators}, Kluwer Acad. Publ., Dordrecht--Boston--London (2000).



\bibitem{Ku-1} \textsc{A.~G.~Êusraev, M.~A.~Pliev}, \emph{Orthogonally additive operators on lattice-normed spaces}, Vladikavkaz Math. J. \No 3 (1999), pp.~33-43.

\bibitem{Ku-2} \textsc{A.~G.~Êusraev, M.~A.~Pliev}, \emph{Weak integral representation of the dominated orthogonally additive operators}, Vladikavkaz Math. J. \No 4 (1999), pp.~22-39.


\bibitem{Maz-1} \textsc{J.~M.~Mazon, S.~Segura de Leon}, \emph{Order bounded ortogonally additive operators}, Rev. Roumane Math. Pures Appl. 35, No~4 (1990), pp.~329-353.

\bibitem{Maz-2} \textsc{J.~M.~Mazon, S.~Segura de Leon}, \emph{Uryson operators}, Rev. Roumane Math. Pures Appl. 35, \No~5 (1990), pp.~431-449.


\bibitem{Pag}\textsc{Pagter, de B}, \emph{The components of a positive operator}, Indag. Math. 48(2) (1983), pp.~229-241.


\bibitem{Pl-3} \textsc{M.~Pliev}, \emph{Uryson operators on the spaces with mixed norm}, Vladikavkaz Math. J. \No 3 (2007), pp.~47-57.

\bibitem{Pl-30} \textsc{M.~Pliev}, \emph{Projection of the positive  Uryson operator}, Vladikavkaz Math. J. \No 4 (2005), pp.~46-51.


\bibitem{Pl-4} \textsc{M.~Pliev}, \emph{Order projections in the space of Uryson operators}, Vladikavkaz Math. J. \No 4 (2006), pp.~38-44.


\bibitem{Pl-5} \textsc{M.~Pliev}, \emph{The shadow of bilinear regular operator}, Vladikavkaz Math. J. \No 3 (2008), pp.~40-45.


\bibitem{Za} \textsc{A.~G.~Zaanen}, \emph{Riesz spaces II}, North Holland, Amsterdam, (1983).


\end{thebibliography}
\end{document}